\documentclass[a4paper]{article}

\usepackage{graphicx}
\usepackage{hyperref}
\usepackage[utf8]{inputenc}
\usepackage{amssymb,amsfonts,amsmath,amsthm}
\usepackage[babel=true,kerning,spacing]{microtype}
\usepackage[english]{babel}
\usepackage[all]{xy}

\newtheorem{theorem}{Theorem}

\newtheorem{proposition}[theorem]{Proposition}
\newtheorem{corollary}[theorem]{Corollary}
\newtheorem{lemma}[theorem]{Lemma}
\newtheorem{definition}[theorem]{Definition}
\theoremstyle{definition}
\newtheorem{example}[theorem]{Example}

\def\A{\mathcal A}

\def\CC{\mathbf C}
\def\FF{\mathbf F}
\def\NN{\mathbf N}
\def\QQ{\mathbf Q}
\def\QQbar{\overline{\QQ}}

\def\ZZ{\mathbf Z}
\def\aaa{\mathfrak a}
\def\bbb{\mathfrak b}

\def\ppp{\mathfrak p}
\def\qqq{\mathfrak q}

\def\CCC{\mathfrak C}

\def\reflexfield{{K^r}}
\def\reflextype{{\Phi^r}}
\newcommand{\kernel}[1]{\Omega_{#1}}
\def\O{R}           
\def\Omin{S}        
\def\Oother{T}      
\def\Osingle{\O'}   
\def\Odouble{\O''}  
\def\Otriple{\O'''} 
\newcommand{\Ominf}[1]{\O_{\mathrm{min}, #1}} 
\def\Omax{{\ZZ_K}}       
\def\Oreal{{\O_0}}       
\def\Orealtotpos{\O_0^{++}}
\def\Ominreal{{\Omin_0}} 
\def\Ootherreal{{\Oother_0}} 
\def\Omaxreal{\ZZ_{K_0}}     
\def\Omaxreflex{\ZZ_{K^r}}

\DeclareMathOperator{\coker}{coker}

\DeclareMathOperator{\id}{id}

\DeclareMathOperator{\lcm}{lcm}
\DeclareMathOperator{\norm}{N}
\DeclareMathOperator{\optcoker}{(co)ker}

\DeclareMathOperator{\val}{val}
\DeclareMathOperator{\End}{End}
\DeclareMathOperator{\Gal}{Gal}

\title{On polarised class groups of orders in quartic CM-fields}
\author{
Gaetan Bisson\thanks{
University of French Polynesia.
\url{http://gaati.org/bisson/}
}
\and
Marco Streng\thanks{
Universiteit Leiden, The Netherlands.
\url{http://www.math.leidenuniv.nl/\textasciitilde streng/}
}
}

\begin{document}

\maketitle

\begin{abstract}

We give an explicit necessary condition
for pairs of orders in a quartic CM-field to have the same
polarised class group.
This generalises a
simpler result for imaginary quadratic fields.

We give an application of our results to
computing endomorphism rings of abelian surfaces over
finite fields,
and we use our results to extend
a completeness result of Murabayashi and Umegaki~\cite{murabayashi-umegaki}
to a list of abelian surfaces over the rationals
with complex multiplication by arbitrary orders.
\end{abstract}

\section{Introduction}\label{sec:intro}

Let $\A$ be a principally polarised abelian surface
defined over a characteristic-zero field~$k$,
and assume that~$\A$ has
\emph{complex multiplication (CM)},
by which we mean that its endomorphism
ring $\O=\End(\A_{\overline{k}})$
is an order in a quartic number field~$K$.
Then $K$ is a CM-field, that is,
a totally imaginary quadratic extension of a totally real field~$K_0$.
The Galois group of the extension $K/K_0$
is generated by
an element of order two, which we denote by $x\mapsto\overline{x}$
and call \emph{complex conjugation}
since it coincides
with complex conjugation for every embedding
$K\rightarrow\CC$.

Our object of study is the \emph{polarised class group} $\CCC(\O)$ of~$\O$ defined as follows.

\begin{definition}\label{def:polarisedgroup}
Let $I_\O$ be the group of pairs $(\aaa,\alpha)$ where $\aaa$ is an
invertible fractional ideal of~$\O$ 
and $\alpha\in K_0$ is a totally positive element satisfying $\aaa\overline\aaa=\alpha\O$.
Let $P_\O$ be the subgroup formed by pairs of the form
$(x\O,x\overline x)$ for $x\in K^\times$. The quotient $I_\O/P_\O$ is called the
\emph{polarised ideal class group} of $\O$ and is written $\CCC(\O)$.
\end{definition}

This group was first introduced by Shimura and Taniyama~\cite[\S 14]{shimura-taniyama}
in the case of the ring of integers~$\O=\ZZ_K$.
It is significant to a number of problems regarding abelian varieties.

For instance, the group $\CCC(\O)$ acts
faithfully on the set of isomorphism classes of principally
polarised abelian surfaces with endomorphism ring $\O$
and the same \emph{CM-type} as $\A$.
Bisson exploits this
to compute endomorphism rings of abelian surfaces~\cite{end-g2}.
This requires telling orders~$\O$ apart using the structure
of their groups $\CCC(\O)$, and our work will allow us to establish in
how far this is possible. See Section~\ref{sec:applicationcompend}.

We also use polarised class groups to extend
van Wamelen's list of principally polarised
abelian surfaces
with conjectural CM by maximal
orders~\cite{wamelen-rationals}.
We extend the list to arbitrary orders and then prove
that it is correct and complete.
The completeness result extends a result of
Murabayashi and Umegaki~\cite{murabayashi-umegaki}
regarding the original list, and
the correctness result is new
in the sense that it completes a partial proof
of van Wamelen~\cite{wamelen-correctness}.
See Section~\ref{sec:applicationwamelen}.

This paper is organised as follows.
First, we introduce the necessary notation and state our results in
Section~\ref{sec:preliminaries}. Section~\ref{sec:ideals-to-elements} then
translates the relevant groups from a setting with ideals into a
setting with elements of finite rings;
this is where most of the work is done in comparing
polarised class groups for varying orders.
In Sections \ref{sec:valuation} and~\ref{sec:maximal},
the work of Section~\ref{sec:ideals-to-elements}
is used in order to derive explicit bounds on the index between orders
that have the same polarised class group.
Finally, Section~\ref{sec:applications} presents our applications mentioned above.

We conclude this section by briefly recalling how the polarised ideal class
group relates to the classical one; although this relationship is not used in
this paper, we hope it might give the reader a better understanding of this
group. Put
$\Oreal=\O\cap K_0$ and
let $\mathrm{Pic}_{\infty}(\Oreal)$
be the narrow class group of $\Oreal$,
that is, the ray class group of the order $\Oreal$
for the infinite modulus.
The order $\O$ is stable under complex conjugation
because of the Rosati involution of $\A$,
hence the norm is a map $N_{K/K_0}:\O\rightarrow \Oreal$.
The sequence
\[ 
\xymatrix@C-0.5pc{
1
\ar[r] &
\Orealtotpos/N_{K/K_0}(\O^\times)
\ar[r] &
\CCC(\O) 
\ar^{\text{forget $\alpha$}}[rr] & &
\mathrm{Pic}(\O)
\ar^{N_{K/K_0}}[rr] & &
\mathrm{Pic}_{\infty}(\Oreal),
}
\]
is exact, 
where the map from the multiplicative
group of totally positive elements
$\Orealtotpos$ to $\CCC(\O)$ is
$\alpha\mapsto ((1), \alpha)$.

\section{Statement of the results}\label{sec:preliminaries}

As before, let $\A$ be a principally polarised abelian surface with complex
multiplication, denote by $\O$ its endomorphism ring, by $K=\QQ\otimes\O$ its
CM-field, and by $K_0$ the totally real subfield of $K$. We note that this
excludes abelian varieties that are not simple over the algebraic closure.

In what follows, we restrict ideals of $I_{\O}$ and $P_{\O}$ to be coprime to
a fixed integer~$f$; this has no effect on the group $\CCC(\O)$,
but it allows us to compare the groups more effectively as the order~$\O$
varies. Indeed,
for any two orders $\Omin\subset \O$, the
invertible ideals of
$\Omin$ coprime to the
index $f=[\Omax:\Omin]$ are in natural bijection with those of $\O$
via the map
$\aaa\mapsto\aaa\O$.
This extends trivially to injections
$I_\Omin\rightarrow I_{\O}$ and $P_\Omin\to P_{\O}$
which yield
a natural morphism $\CCC(\Omin)\to\CCC(\O)$. We will use these maps
implicitly and say for instance that an ideal $\aaa\subset \Omin$ is principal in $\O$
when we actually mean that $\aaa\O$ is.

To compare $\CCC(\Omin)$ and $\CCC(\O)$ in a computationally efficient manner,
the algorithm of Bisson \cite{end-g2} uses certain elements of $\CCC(\Omin)$ 
best generated through the \emph{reflex type norm} map.
At the same time, the complex multiplication theory that describes
the field of moduli of $\A$ is also phrased in terms of this map.
For these reasons, our main results are formulated in terms of this
map, which we now define.

In the following, let $K$ be a CM-field of arbitrary degree.
A CM-type~$\Phi$ of~$K$ with values in~$\CC$
is a complete set of representatives for the
embeddings $K\rightarrow \CC$ up to complex conjugation.
The \emph{reflex field} $\reflexfield\subset \CC$ of~$\Phi$ is the 
subfield generated by the image
of the \emph{type norm}
\[
\norm_\Phi:K\rightarrow \CC : x\mapsto\prod_{\phi\in\Phi}\phi(x).
\]
The type norm is a half-norm in the sense that
$\norm_\Phi(x)\overline{\norm_\Phi(x)} = N_{K/\QQ}(x)$.

\begin{example}\label{ex:concretecm1}
CM-fields of degree $4$ are either cyclic Galois,
non-Galois with Galois group $D_4$ of order~$8$,
or biquadratic Galois (with Galois group $C_2\times C_2$).
It is known that in the biquadratic Galois case, the abelian surface
is non-simple and has a matrix ring as endomorphism ring,
so this paper restricts to the non-biquadratic case.

In the $C_4$-case, let $\mathrm{Gal}(K/\QQ)=\langle\rho\rangle$
and, in the $D_4$-case, let $\mathrm{Gal}(K^c/K)=\langle\sigma\rangle$
and $\mathrm{Gal}(K^c/\QQ)=\langle\sigma,\rho\rangle$
with $\rho^4=\sigma^2=\sigma\rho\sigma\rho = \id$,
where $K^c$ denotes the normal closure of $K$.
In both cases, we have $\rho^2 = \overline{\cdot}$
and there exists an embedding
$\phi:K^c\rightarrow \CC$
such that $\Phi = \{\phi, \phi\circ\rho\}$.
In the $C_4$-case, we have $\reflexfield = \phi(K)\cong K$.
In the non-Galois case, we have
$\reflexfield = \phi(K')\not\cong K$, where $K'\subset K^c$ is the fixed
field of $\langle \rho\sigma\rangle\subset \mathrm{Gal}(K^c/\QQ)$.
\end{example}

The \emph{reflex type}~\cite[\S 8]{shimura-taniyama} is the set
$\reflextype=\{\psi^{-1}_{|\reflexfield} : \psi\in S\}$
of embeddings of~$\reflexfield$ into the normal closure~$K^c$ of~$K$,
where~$S$ is the set of all embeddings of $K^c$ into~$\CC$
that extend elements of~$\Phi$.
\begin{example}\label{ex:concretecm2}
Concretely in the quartic non-biquadratic situation above,
we have a map $\phi^{-1} : \reflexfield \rightarrow K^c$
and $\reflextype = \{\phi^{-1}, \rho^{-1}\circ\phi^{-1}\}$.
\end{example}

The reflex type is a CM-type of $\reflexfield$ with values in $K^c$ and reflex field~$K$,
hence defines a type norm $\norm_{\reflextype}:\reflexfield\to K$. This norm
can be extended into a homomorphism from the group $I_{\reflexfield}=I_{\reflexfield}(f)$
of invertible ideals~$\aaa$ of $\O_\reflexfield$ coprime to~$f$ to 
the analogous group for~$\O$.
In turn, this defines a map
\begin{align}
I_{\reflexfield}(f) & \longrightarrow \CCC(\O),\label{eq:reflexnormccc}\\
\aaa &\longmapsto
(\norm_{\reflextype}(\aaa),\norm_{\reflexfield/\QQ}(\aaa)).\nonumber
\end{align}
We denote its kernel by~$\kernel{\O}$. It is an important object
because the image $I_{\reflexfield}(f)/\kernel{\O}$ appears naturally as the Galois
group of the field of moduli of the abelian surfaces of type $\Phi$
with endomorphism ring~$\O$ \cite[Main Theorem 3 on page 142]{shimura-taniyama}.
Our main results are the following two theorems,
together with their ingredients from Section~\ref{sec:ideals-to-elements}
and applications
in Section~\ref{sec:applications}.

\begin{theorem}\label{thm:relativeindex}
For any order~$\Oother$ stable under complex conjugation
in a quartic non-biquadratic
CM-field~$K\not\cong\QQ(\zeta_5)$ satisfying
$\kernel{\Omax}\subset \kernel{\Oother}$
we have
\[
[\Omax:\Oother]^2\ \mid\ 2^{40} 3^{16} N_{K_0/\QQ}(\Delta_{K/K_0}).
\]
\end{theorem}

Throughout this paper, we
let $\O$ and $\Oother$ be any two orders
stable under complex conjugation
in a quartic non-biquadratic CM-field~$K$.
We define $\Omin=\O\cap\Oother$
and denote the intersection of any of these orders
with the totally real subfield $K_0$
(which gives an order in $K_0$)
by adding a subscript $0$ to it.

\begin{theorem}\label{thm:firstmaintheorem}
Let the notation be as above.
If $\kernel{\O}\subset \kernel{\Oother}$ and $\O\not\cong \ZZ[\zeta_5]$,
then the quotient $[\O:\Omin]/[\Oreal:\Ominreal]$
is an integer dividing~$2^{10} 3^4$.
\end{theorem}

We do not claim that the constants are optimal, but we do show in
Example~\ref{ex:realindexrequired} that the index $[\Oreal:\Ominreal]$ is
necessary in Theorem~\ref{thm:firstmaintheorem}. On the other hand, it remains
an open question whether the discriminant factor is necessary in
Theorem~\ref{thm:relativeindex}.

For the excluded case $\O=\ZZ[\zeta_5]$,
we show
in Section~\ref{sec:applicationwamelen}
that if $\kernel{\ZZ[\zeta_5]}\subset \kernel{\Oother}$,
then $[\ZZ[\zeta_5]:\Oother]$ divides one of
$2^4$, $3^2$ and $5^2$.

\theoremstyle{plain}
\newtheorem*{maintheoremgeneral}{{Theorem~\ref{thm:firstmaintheorem}}}
\newtheorem*{maintheoremmaximal}{{Theorem~\ref{thm:relativeindex}}}

\section{From ideals to elements}\label{sec:ideals-to-elements}

If $\O$ and $\Oother$ are two orders of a quartic non-biquadratic
CM-field~$K$,
we obviously have the
implication $\O\subset\Oother\Rightarrow \kernel{\O}\subset \kernel{\Oother}$. The goal of
Section~\ref{sec:valuation}
is to establish a partial converse to this statement, that is, to obtain an
explicit condition on $\O$ and $\Oother$ necessary for $\kernel{\O}\subset \kernel{\Oother}$ to
hold.
To help with that, we first express some relevant groups in terms of ring
\emph{elements}, as opposed to ideals.
Recall that $\Omin$ denotes the intersection of $\O$ and~$\Oother$,
and fix $f=[\Omax:\Omin]$ once and for all.

\begin{lemma}\label{lemma:squares}
All squares of elements of
$\CCC(\Omin)$ are in the image of 
the type norm map~\eqref{eq:reflexnormccc}.
\end{lemma}
\begin{proof}
The key to this result is the following observation, which is also 
\cite[Lemma 2.6]{streng-reciprocity}
and, in the case of maximal orders, \cite[Lemma I.8.4]{streng-thesis}.
Let $\tau=\rho_{|K_0}$ be the nontrivial automorphism of~$K_0$.
For every invertible $\O$-ideal $\aaa$, we have
$N_{\reflextype}(N_{\Phi}(\aaa)) = \aaa^2\tau (\aaa\overline{\aaa})$.
The proof of this follows directly from Examples \ref{ex:concretecm1}
and~\ref{ex:concretecm2} since
\[ N_{\reflextype}(N_{\Phi}(\aaa)) = {}^{(1+\rho^{-1})\phi^{-1}}({}^{\phi(1+\rho)}\aaa)
= {}^{(2+\rho^{-1}+\rho)}\aaa = \tau(\aaa\overline{\aaa})\aaa^2.\]

  For any $(\aaa, \alpha)\in I_{\Omin}$ we have
\begin{equation}\label{eq:reflex-norm-square}
    (\aaa, \alpha)^2 = (\tau(\alpha)\Omin, (\tau(\alpha))^2)^{-1}
                 (\aaa^2\tau(\aaa\overline{\aaa}), N_{K_0/\QQ}(\alpha)^2).
\end{equation}
  The first factor is trivial in $\CCC(\Omin)$, and 
the second factor is obtained from $\aaa$ via the map
of~\eqref{eq:reflexnormccc}.
  This proves that the class of $(\aaa, \alpha)^2$ is in the image.
\end{proof}

By Lemma~\ref{lemma:squares}, a necessary condition for $\kernel{\O}\subset \kernel{\Oother}$ is
\begin{equation}\label{eq:kernels-contained}
\ker(\CCC(\Omin)^2\to\CCC(\O))\quad\subset\quad\ker(\CCC(\Omin)^2\to\CCC(\Oother)),
\end{equation}
and, to write these kernels in terms of ring elements,
we use the following lemma.

\begin{lemma}\label{lem:kernels-contained}
Denote by $\phi$ the natural morphism $\CCC(\Omin)\rightarrow \CCC(\O)$
and consider the relative norm
\begin{equation}\label{eq:psi}
 \psi: \frac{(\O / f \Omax)^\times}{(\Omin / f \Omax)^\times \mu_{\O}}
\longrightarrow \frac{(\Oreal / f \Omaxreal)^\times}{(\Ominreal / f \Omaxreal)^\times}.
\end{equation}
We have $\ker\phi = \ker\psi$ and $\coker\phi\subset \coker\psi$.
\end{lemma}

To prove the above we first require a simple technical result.

\begin{lemma}\label{lem:coprime}
An ideal $\aaa$ of $\Omin$ is coprime to $f$ (that is, it satisfies
$\aaa+f\Omin=\Omin$) if and only if it satisfies $\aaa+f\Omax=\Omin$.
\end{lemma}

\begin{proof}
Recall that $f\Omax\subset\Omin$. The ``only if'' part is trivial. Now for
the converse, suppose that $\aaa$ satisfies $\aaa+f\Omax=\Omin$; it follows
that $(\aaa+f\Omax)^2=\Omin$. Since $(\aaa+f\Omax)^2\subset \aaa+(f\Omax)^2$
and $(f\Omax)^2\subset f\Omin$, we deduce that $\aaa+f\Omin=\Omin$.
\end{proof}

We may now prove Lemma~\ref{lem:kernels-contained} and make free and implicit
use of Lemma~\ref{lem:coprime}.

\begin{proof}[Proof of Lemma~\ref{lem:kernels-contained}.]
Recall $\Oreal=\O\cap K_0$ and $\Ominreal=\Omin\cap K_0$.
The diagram
\[
\xymatrix{
1 \ar[r] & P_{\Omin} \ar[r]\ar@{^{(}->}[d]
                  & I_{\Omin} \ar[r]\ar@{^{(}->}[d]
                           & \CCC(\Omin) \ar[r]\ar_{\phi}[d] & 1\\
1 \ar[r] & P_{\O} \ar[r]
                  & I_{\O} \ar[r]
                           & \CCC(\O) \ar[r] & 1
}
\]
has exact rows and, as we restrict to ideals coprime to $f$,
its two leftmost vertical arrows are injective.
The snake lemma thus tells us that
\[\optcoker\phi = \optcoker \left(\frac{P_{\O}}{P_{\Omin}}\rightarrow \frac{I_{\O}}{I_{\Omin}}\right).\]
It now suffices to give a natural isomorphism and an embedding
\[\frac{P_{\O}}{P_{\Omin}} =
\frac{(\O / f \Omax)^\times}{(\Omin / f \Omax)^\times \mu_{\O}},
\quad\quad
\frac{I_{\O}}{I_{\Omin}} \hookrightarrow
\frac{(\Oreal / f \Omaxreal)^\times}{(\Ominreal / f \Omaxreal)^\times}\]
such that the induced map is~$\psi$. For these maps, we take
$(x\O, x\overline{x})\mapsto x$
(for integral representatives $(x\O, x\overline{x})$
of elements of $P_{\O}/P_{\Omin}$)
 and $(\aaa, \alpha)\mapsto \alpha$
(for integral representatives $(\aaa,\alpha)$
of elements of $I_{\O}/I_{\Omin}$).

On $P_{\O}$, the first map is well-defined, as $(x\O, x\overline{x})$
determines $x$ up to roots of unity in~$\O$.
The kernel then consists of those pairs $(x\O, x\overline{x})$
with $x$ in $\Omin$ invertible modulo~$f$,
that is, such that
$x\Omin$ is coprime to $f$.
In other words, the kernel is $P_{\Omin}$,
so the map is indeed well-defined and
injective on $P_\O/P_\Omin$. Surjectivity is obvious.

The second map is well-defined on $I_\O/I_\Omin$.
Now suppose that $(\aaa, \alpha)$ is in the kernel,
again without loss of generality with~$\aaa$ integral.
Then
$\alpha\Ominreal$ is coprime to~$f$, so
$\alpha\Omin$ is also coprime to~$f$.
Denoting by $\bbb$ the
unique $\Omin$-ideal
coprime to~$f$
satisfying $\aaa = \bbb\O$,
we have $\bbb\overline{\bbb}\O = \alpha\O$
and, as $\alpha\Omin$ and $\bbb$ are coprime to~$f$, we find
$(\bbb,\alpha)\in I_{\Omin}$ which implies
$(\aaa,\alpha)\in I_{\Omin}$ by abuse of notation.
This proves injectivity.

Finally, the induced map $\psi$ is indeed the relative norm,
as $x$ maps via $(x\O,x\overline{x})$ to $x\overline{x}$.
\end{proof}

Using Lemma~\ref{lem:kernels-contained},
the necessary condition \eqref{eq:kernels-contained} becomes as follows.

\begin{proposition}\label{prop:kernel}
Let $\O$ and $\Oother$ be orders in a quartic non-biquadratic CM-field~$K$
and let $\Omin=\Oother\cap \O$.
If we have $\kernel{\O}\subset \kernel{\Oother}$,
then the kernel
\begin{equation}\label{eq:kernel}
\ker \left(\psi : \frac{(\O/f\Omax)^\times}{(\Omin/f\Omax)^\times\mu_\O}
\rightarrow
\frac{(\Oreal / f \Omaxreal)^\times}{(\Ominreal / f \Omaxreal)^\times}
\right)\end{equation}
is of exponent at most two.
\end{proposition}

\begin{proof}

By \eqref{eq:kernels-contained}, if $\kernel{\O}\subset \kernel{\Oother}$, then we have
\begin{align*}\ker\left(\CCC(\Omin)\to\CCC(\O)\right)^2
&\subset
\ker\left(\CCC(\Omin)^2\to\CCC(\O)\right)\\
&\subset
\ker\left(\CCC(\Omin)^2\to\CCC(\Oother)\right)
\subset
\ker\left(\CCC(\Omin)\to\CCC(\Oother)\right),\end{align*}
hence
\begin{equation}\label{eq:squaresofkernels}
\ker\left(\CCC(\Omin)\to\CCC(\O)\right)^2
\quad\subset\quad \ker\left(\CCC(\Omin)\to\CCC(\Oother)\right)\cap
\ker\left(\CCC(\Omin)\to\CCC(\O)\right).
\end{equation}

We wish to apply Lemma~\ref{lem:kernels-contained} to both $\O$ and $\Oother$ and compare the results,
so we need some common supergroup for \[\frac{(\O/f\Omax)^\times}{(\Omin/f\Omax)^\times\mu_{\O}}\quad
\mbox{and}\quad \frac{(\Oother/f\Omax)^\times}{(\Omin/f\Omax)^\times\mu_{\Oother}}\]
to compare them in.
We start by proving that the natural map 
\[ \iota_\O:\frac{(\O/f\Omax)^\times}{(\Omin/f\Omax)^\times\mu_{\O}}\rightarrow \frac{(\Omax/f\Omax)^\times}{(\Omin/f\Omax)^\times\mu_{\Omax}}\]
is injective and compatible with the natural isomorphisms to $P_{\O}/P_{\Omin}$ and $P_{\Omax}/P_{\Omin}$
of the proof of Lemma~\ref{lem:kernels-contained}.
The latter is trivial, as we take ideals coprime to $f\Omax$ in~$K$.
The former then follows by injectivity of the inclusion $P_{\O}/P_{\Omin}\subset P_{\Omax}/P_{\Omin}$.
%

Next, note that we have $\mathrm{im}(\iota_\O)\cap\mathrm{im}(\iota_{\Oother}) = 1$.
Indeed, if $\overline{x}=(x+f\Omax)\in(\Omax/f\Omax)^*$ represents
an element $X$ of the intersection, then $\overline{x} \in (\O+f\Omax)^*$ and $\overline{x}\in (\Oother+f\Omax)^*$,
so $\overline{x}\in (S+f\Omax)^*$, hence $X$ is trivial.

Applying Lemma~\ref{lem:kernels-contained} to both~$\O$ and~$\Oother$,
equation~\eqref{eq:squaresofkernels} becomes
\begin{align*}
& \ker\left(\frac{(\O/f\Omax)^\times}{(\Omin/f\Omax)^\times\mu_\O}
\rightarrow 
\frac{(\Oreal / f \Omaxreal)^\times}{(\Ominreal / f \Omaxreal)^\times}
\right)^2
\\
\subset \quad &
\ker \left(\mathrm{im}(\iota_\O)
\rightarrow 
\frac{(\Omaxreal / f \Omaxreal)^\times}{(\Ominreal / f \Omaxreal)^\times}
\right)\cap
\ker \left(\mathrm{im}(\iota_\Oother)
\rightarrow 
\frac{(\Omaxreal / f \Omaxreal)^\times}{(\Ominreal / f \Omaxreal)^\times}
\right)
\\
\subset \quad &
\ker \left(\left(\mathrm{im}(\iota_\O) \cap \mathrm{im}(\iota_\Oother)\right)
\rightarrow 
\frac{(\Omaxreal / f \Omaxreal)^\times}{(\Ominreal / f \Omaxreal)^\times}
\right) = 1.\qedhere
\end{align*}

\end{proof}

Note that, unless $\O\cong \ZZ[\zeta_5]$, the unit group $\mu_\O = \{\pm 1\}$ can
be absorbed into $(\Omin / f\Omax)^\times$.

\section{Explicit bounds}\label{sec:valuation}

We shall now derive from Proposition~\ref{prop:kernel} a more explicit
necessary condition for $\kernel{\O}\subset \kernel{\Oother}$ on the indices of the relevant
orders. First, let us give a weak but natural result bounding the size of the
quotient groups of \eqref{eq:kernel} in terms of these indices.

\begin{proposition}\label{prop:weak}
Let $\Omin\subset\O$ be two orders in a number field $K$ of degree $n$,
and let $f$ be a multiple of the index $[\O:\Omin]$; we have
\[
[\O:\Omin]\prod_{p\mid f}(1-1/p)^n
\leq
\left|\frac{(\O/f\O)^\times}{(\Omin/f\O)^\times}\right|
\leq
[\O:\Omin]\prod_{p\mid f}(1-1/p)^{-n}.
\]
\end{proposition}

\begin{proof}
Decomposing the ring $\O/f\O$ over prime ideals $\ppp$ of $\O$ dividing $f$ yields
\[
\left|(\O/f\O)^\times\right| = \left|(\O/f\O)\right| \prod_\ppp(1-1/\norm(\ppp)).
\]
Since there are at most $n$ such $\ppp$ dividing each prime factor $p$ of~$f$, we derive
\[
\frac{\left|(\O/f\O)^\times\right|}{\left|(\Omin/f\O)^\times\right|}
\geq
\frac{\left|(\O/f\O)\right|}{\left|(\Omin/f\O)\right|} \prod_\ppp(1-1/\norm(\ppp))
\geq
[\O:\Omin]\prod_{p\mid f}(1-1/p)^n
\]
and the second inequality follows similarly.
\end{proof}


Proposition~\ref{prop:weak} immediately leads
to the following bounds.
\begin{lemma}\label{lem:valuation}
Given a quartic CM-field~$K$
and orders as in the previous sections,
let $v=\mathrm{val}_p([\O:\Omin]/[\Oreal:\Ominreal])$.
If the group~\eqref{eq:kernel} is of exponent one or two,
we have $p^{v-6}(p-1)^6\leq 2^3\#\mu_\O$.

In particular, if $\O\not\cong\ZZ[\zeta_5]$,
then for~$p=2$ we have~$v\leq 10$,
and for~$p=3$ we have~$v\leq 4$.
Finally, if~$v\geq 1$, then~$p < 10\#\mu_\O$.
\end{lemma}
\begin{proof}
By Proposition~\ref{prop:weak}, the domain
in~\eqref{eq:kernel} has order greater than or equal to
$$[\O:\Omin]\frac{(1-1/p)^4}{\#(\mu_{\O}/\{\pm 1\})},$$ 
while the codomain has order at most
$[\Oreal:\Ominreal](1-1/p)^{-2}$,
hence the kernel has order
$\geq p^v(1-1/p)^6/\#(\mu_{\O}/\{\pm 1\})$.
On the other hand, the kernel is generated
by a set of at most four elements, so,
if it has exponent at most~$2$, then
its order is at most~$2^4$,
which yields the first bound.

The other bounds are specialisations of this
first bound, using the fact that $\O\not\cong\ZZ[\zeta_5]$
implies $\mu_\O = \{\pm 1\}$.
\end{proof}
The bounds on $v$ of Lemma~\ref{lem:valuation}
can easily be sharpened with elementary
observations about the groups involved,
but it is hard to make
them completely sharp, so we leave them as they are.
Instead, we sharpen the bound on the prime~$p$
by looking more closely at the structure of the groups.

\begin{proposition}\label{prop:valuation}
If the kernel \eqref{eq:kernel} contains no element of order greater than
two, then for every prime~$p$ we have
$v:=\val_p([\O:\Omin]/[\Oreal:\Ominreal])=0$,
except possibly for $p\leq 3$,
and except possibly for $p\leq 19$ if $\O=\ZZ[\zeta_5]$.
\end{proposition}

\begin{proof}
Assume that the kernel
\eqref{eq:kernel} is of exponent one or two,
and suppose that there exists a prime $p\geq 5$ at which the
quotient $[\O:\Omin]/[\Oreal:\Ominreal]$ is nontrivial.

To better work with this kernel, we first write the domain and codomain of the
relative norm map of \eqref{eq:kernel} more explicitly by decomposing the
$p$-part of the ring $\O/f\Omax$ over primes ideals $\qqq$ of $\O$ dividing
$(p)=p\O$; this gives
\[
(\O/(f\Omax)_{(p)})^\times=\prod_{\qqq}(\O/(f\Omax)_{(\qqq)})^\times
\cong
\prod_{\qqq}(\O/\qqq)^\times \times \frac{1+\qqq}{1+(f\Omax)_{(\qqq)}},
\]
where $I_{(\qqq)}$ denotes the $\qqq$-primary part of $I$,
that is, $I_{(\qqq)}=(I\cdot \O_{\qqq})\cap\O=I+\qqq^n$ for all
sufficiently large~$n$ (see \cite{stevenhagen-rings}).
Therefore, omitting the case $\O\cong\ZZ[\zeta_5]$ for now, we
find that dividing out by
$\mu_\O=\{\pm 1\}$ is trivial and the domain from \eqref{eq:kernel} can be written
locally at~$p$ as
\[
D:=\frac{(\O/(f\Omax)_{(p)})^\times}{(\Omin/(f\Omax)_{(p)})^\times}\cong
\frac{\prod_\qqq (\O/\qqq)^\times}{\prod_\ppp(\Omin/\ppp)^\times}
\times A_p;
\]
for some $p$-group~$A_p$.
Similarly, the codomain can be written locally at $p$ as
\[
C:=\frac{(\Oreal/(f\Omaxreal)_{(p)})^\times}{(\Ominreal/(f\Omaxreal)_{(p)})^\times}\cong
\frac{\prod_{\qqq_0} (\Oreal/\qqq_0)^\times}{\prod_{\ppp_0}(\Ominreal/\ppp_0)^\times}
\times A_{0,p}.
\]

The rightmost factors $A_{0,p}\subset A_p$ are 
$p$-groups and, as $p$ is
odd and the kernel~\eqref{eq:kernel} has exponent at most two,
they are equal.

Since $\Oreal$ and $\Ominreal$ are quadratic,
the non-$p$-part of the codomain
$C$ is cyclic and of order $1$, $p-1$ or $p+1$.
Correspondingly, the $p$-part
$[\Oreal:\Ominreal]_p$ of the index is $\#A_{0,p}$ in the first case
and $p\cdot \#A_{0,p}$ in the
other cases.

On the other hand, the domain
$D$ is a product of cyclic
groups with orders of the form $p^e$, $p^f-1$ and $(p^g-1)/(p^h-1)$ where the
exponents $f, g, h$, are in $\{1,2,4\}$ and $h<g$.
The valuation at $p$ of the index
$[\O:\Omin]$ is then the sum of all $e$'s, $f$'s and $(g-h)$'s.
As the exponent of the kernel is at most two, the order of
each of the coprime-to-$p$
cyclic factors must divide the non-$p$-part of $2\#C$,
which is $2$, $2(p-1)$ or $2(p+1)$.

As $p\geq 5$, it is easy to see that $p^f-1$ and $(p^g-1)/(p^h-1)$
do not divide $2\#C$, except when they are equal to
$\#C/\#A_{0,p}\in\{p-1,p+1\}$.
As these numbers are greater than three,
this observation also shows that if there were multiple such cyclic factors in~$D$,
it would contradict the fact that a quotient by a $2$-torsion subgroup
is contained in~$C$.

In particular, we get $\val_{p}([\O:\Omin])\in\{\#A_p, \#A_{p}+1\}$,
where the second case is only possible when $[\Oreal:\Ominreal]=\#A_{0,p}+1$.
We thus conclude $\val_{p}([\O:\Omin])=\val_{p}([\Oreal:\Ominreal])$.

\bigskip

It remains to consider the case $\O\cong \ZZ[\zeta_5]$ where
$\mu_\O\simeq\ZZ/10\ZZ$.
In this case, the orders of the cyclic factors of $D$ must divide
$10\# C$, and the proof goes through for $p>19$.
\end{proof}

\begin{maintheoremgeneral}
If $\kernel{\O}\subset \kernel{\Oother}$ and $\O\not\cong \ZZ[\zeta_5]$,
then the quotient $[\O:\Omin]/[\Oreal:\Ominreal]$
is an integer dividing~$2^{10} 3^4$.
\end{maintheoremgeneral}
\begin{proof}
This is a combination of Lemma~\ref{lem:valuation}
and Propositions \ref{prop:kernel} and~\ref{prop:valuation}.
\end{proof}

\section{Stronger result in the maximal case}\label{sec:maximal}

The result of Proposition~\ref{prop:valuation}, namely that $\val_p
[\O:\Omin]=\val_p [\Oreal:\Ominreal]$ for all but finitely many primes $p$,
may not seem
very strong. Only when $\Oreal=\Ominreal$
does it immediately imply that $\O$ and $\Omin$
are identical locally at all $p>3$.
The goal of this section is to prove that, in the case $\O=\Omax$
and $\Oother=\Omin$,
Proposition~\ref{prop:valuation} further implies
that $\O$ and $\Oother$ are identical locally
if we rule out a few more primes~$p$.

\begin{maintheoremmaximal}
For any order~$\Omin$ stable under complex conjugation
in a quartic non-biquadratic
CM-field~$K\not\cong\QQ(\zeta_5)$ satisfying
$\kernel{\Omax}\subset \kernel{\Omin}$
we have
\[
[\Omax:\Omin]^2\ \mid\ 2^{40} 3^{16} N_{K_0/\QQ}(\Delta_{K/K_0}).
\]
\end{maintheoremmaximal}

To prove the theorem, we first
show that $[\Omaxreal:\Ominreal]^2$
almost divides $[\Omax:\Omin]$.

\begin{lemma}\label{lem:relativeindex}
Let $K_0/\QQ$ be a quadratic field
and $K/K_0$ a finite Galois extension of degree~$n$.
Let $\Omin$ be an order of~$K$ stable under
$\Gal(K/K_0)$, and let $\Ominreal=\Omin\cap K_0$.
We have
\begin{equation}\label{eq:relativeindex}
[\Omaxreal:\Ominreal]^{2n}\ \mid
\ N_{K_0/\QQ}(\Delta_{K/K_0})\ 
[\Omax:\Omin]^2.
\end{equation}
\end{lemma}

\begin{proof}
Write the orders as $\Omaxreal=\ZZ+\omega\ZZ$ and $\Ominreal=\ZZ+c\omega\ZZ$
where $c=[\Omaxreal : \Ominreal]$,
and let $\delta = 2\omega-\mathrm{tr}_{K_0/\QQ}(\omega)$.
Note that~$(\delta)$ is the different of~$K_0$.

First of all, we have $\mathrm{tr}_{K/\QQ}(\delta^{-1}\Omin)
=\mathrm{tr}_{K_0/\QQ}(\mathrm{tr}_{K/K_0}(\delta^{-1}\Omin))$.
Using the fact that $K/K_0$ is Galois and $\Omin$ is
stable under Galois, we find
$\mathrm{tr}_{K/K_0}(\Omin)\subset\Ominreal$,
so
$\mathrm{tr}_{K/\QQ}(\delta^{-1}\Omin)
\subset \mathrm{tr}_{K_0/\QQ}(\delta^{-1}\Ominreal)
= c\ZZ$. In particular, $(c\delta)^{-1}\Omin$
is contained in the trace dual $\Omin^*$ of~$\Omin$,
so the following index is an integer:
\begin{align*}
[\Omin^* : (c\delta)^{-1}\Omin]
 &= N_{K/\QQ}(c\delta)^{-1} [\Omin^*:\Omax^*] [\Omax^* : \Omax] [\Omax:\Omin]\\
 &= c^{-2n} \Delta_{K_0/\QQ}^{-n}
[\Omin^*:\Omax^*] \Delta_{K/\QQ} [\Omax:\Omin]\\
 &= c^{-2n} N_{K_0/\QQ}(\Delta_{K/K_0}) 
[\Omin^*:\Omax^*]
[\Omax:\Omin].\end{align*}
Linear algebra gives us $[\Omin^*:\Omax^*]
=[\Omax:\Omin]$, and the result follows.
\end{proof}

%

\begin{proof}[{Proof of Theorem~\ref{thm:relativeindex}}]
Let $\Omin=\Oother$.
Then Theorem~\ref{thm:firstmaintheorem}
and Lemma~\ref{lem:relativeindex} give
$$[\Omax:\Omin]^4\ \mid \ (2^{10} 3^4)^4 [\Omaxreal:\Ominreal]^4
\ \mid \ 
2^{40} 3^{16} N_{K_0/\QQ}(\Delta_{K/K_0}) [\Omax:\Omin]^2.$$
Dividing both sides by $[\Omax:\Omin]$ yields the result.
\end{proof}

\begin{example}
The factor $N_{K_0/\QQ}(\Delta_{K/K_0})$ on the right
hand side of \eqref{eq:relativeindex}
cannot be omitted. Consider for instance the order
\[
\Omin = \ZZ[5\sqrt{7}] + \sqrt{5\cdot(-3+\sqrt{7})}\ZZ[\sqrt{7}]
\]
of which the real order is $\Ominreal = \ZZ[5\sqrt{7}]$.
The indices $[\Omax:\Omin]$ and $[\Omaxreal:\Ominreal]$ are both~$5$,
so Lemma~\ref{lem:relativeindex} would be false without that factor.
\end{example}

\begin{example}\label{ex:realindexrequired}
The index $[\Oreal:\Ominreal]$ is required in Theorem~\ref{thm:firstmaintheorem}.
Indeed, let $K$ be any quartic CM-field with $\Omax=\ZZ[\beta]$, where
$\beta$ is a square root of a totally negative number in $K_0$;
fix a positive odd integer $F$ and let
\[
\begin{array}{l@{}c@{}r@{}r@{}r@{}r}
\O^{\phantom{\prime}}\  & =\  \ZZ+  F^2\beta\ZZ +  F^2\beta^2\ZZ +  F^2\beta^3\ZZ; \\
\O^{\prime}\ & = \ \ZZ+  F^2\beta\ZZ +  {}^{\phantom{2}}F \beta^2\ZZ+  F^2\beta^3\ZZ.
\end{array}
\]
Note $\O\subset\Oother$. The corresponding real orders are
\[
\begin{array}{l@{}c@{}r@{}r}
\Oreal \ & = \ \ZZ+  F^2\beta^2\ZZ; \\
\Ootherreal \ & = \ \ZZ+  {}^{\phantom{2}}F  \beta^2\ZZ.
\end{array}
\]
We have $(\O/F^2\Omax)=(\Oreal/F^2\Omaxreal)$ and
$(\Oother/F^2\Omax)=(\Ootherreal/F^2\Omaxreal)$; their respective
unit groups have order $(F-1)F^3$ and $(F-1)F$.
Thus, the map $\psi$ from Proposition~\ref{prop:kernel} with $f=F^2$ maps $x$
to $x\overline{x}=x^2$ on a group of odd order $F^2$;
therefore we have $\ker(\psi)=1$ which implies $\kernel{\O}=\kernel{\Oother}$.

Concretely, the field $K=\QQ(\beta)$ with
$\beta = \sqrt{-3+\sqrt{2}}$ satisfies $\Omax=\ZZ[\beta]$.
But really the assumption
$\Omax = \ZZ[\beta]$ is only there to simplify
the argument.
\end{example}

\section{Applications}\label{sec:applications}

\subsection{Abelian surfaces with complex multiplication over the rationals}
\label{sec:applicationwamelen}

\subsubsection{Statements}

Van Wamelen \cite{wamelen-rationals} gives a conjectural list
of curves of genus two defined over $\QQ$ with complex multiplication by maximal orders.
Our results allow us to finish the proof of 
this list, as well as generalise it
to arbitrary orders.
\begin{theorem}[Extending results of Murabayashi-Umegaki~\cite{murabayashi-umegaki} and Van Wamelen~\cite{wamelen-correctness}]\label{thm:wamelenlist}
The $19$ curves given in \cite{wamelen-rationals} are 
(up to $\overline{\QQ}$-isomorphism) exactly the
curves $C/\QQ$ of genus two with $\mathrm{End}(J(C)_{\overline{\QQ}})\cong\Omax$ for a quartic 
CM-field~$K$.
\end{theorem}
\begin{theorem}\label{thm:classification2}
The curves
\begin{align*}
C &:  y^2 = x^6 - 4x^5 + 10x^3 - 6x - 1\qquad\qquad\qquad\mbox{and}\\
D &:  y^2 = 4x^5 + 40x^4 - 40x^3 + 20x^2 + 20x + 3
\end{align*}
have endomorphism rings
\begin{align*}
\mathrm{End}(J(C)_{\overline{\QQ}})&\cong
\ZZ + 2 \zeta_{5}\ZZ + (\zeta_{5}^{2} + \zeta_{5}^{3})\ZZ + 2 \zeta_{5}^{3}\ZZ\qquad\qquad{\mbox{and}}\\
\mathrm{End}(J(D)_{\overline{\QQ}}) & \cong
\ZZ + (\zeta_5 + 3 \zeta_{5}^{3})\ZZ + (\zeta_5^2 + \zeta_{5}^{3})\ZZ + 5 \zeta_{5}^{3}\ZZ.
\end{align*}
Moreover, they are (up to
$\overline{\QQ}$-isomorphism) the only curves of genus two
with field of moduli~$\QQ$
such that $\mathrm{End}(J(C)_{\overline{\QQ}})$ is
a \emph{non}-maximal order in
any of the fields in Tables~\ref{tab:provingwamelen}
and~\ref{tab:otherfields} below.
\end{theorem}

P{\i}nar K{\i}l{\i}{\c{c}}er, with the second-named author,
proves
that all cyclic quartic
CM-fields~$K$ with $\kernel{\Omax}=I_{\reflexfield}$ appear
in Tables~\ref{tab:provingwamelen}
and~\ref{tab:otherfields}.
In particular, 
the $21$ curves of Theorems \ref{thm:wamelenlist}
and~\ref{thm:classification2}
are 
(up to $\overline{\QQ}$-isomorphism) exactly the
curves $C/\QQ$ of genus two 
such that
$\mathrm{End}(J(C)_{\overline{\QQ}})\otimes \QQ$
is a quartic field.
See~\cite{kilicer-streng}.

We give the proofs of Theorems \ref{thm:wamelenlist}
and~\ref{thm:classification2}
in Sections \ref{sec:provingwamelen} and~\ref{sec:evidence}.

\subsubsection{Background}

We start by explaining what in Theorem~\ref{thm:wamelenlist}
was already proven, and what remained to be.
Murabayashi and Umegaki~\cite{murabayashi-umegaki}
prove that the $13$ fields in Table~\ref{tab:provingwamelen}
are the only quartic fields whose maximal orders are
endomorphism
rings of genus-two curves over~$\QQ$.

Van Wamelen~\cite{wamelen-rationals} computes that each of the $13$ fields in this list
has $1$ or $2$ curves corresponding to it, and determines these
curves numerically to high precision.
This yields his list of $19$ curves referenced in the theorem.
Van Wamelen~\cite{wamelen-correctness} later proved that each
of his $19$ curves does have complex multiplication, and though he does not
prove that the endomorphism ring is the \emph{maximal} order, he
suggests how to prove this by numerical approximation.
We will use our methods to finish the proof
that the order is maximal while avoiding numerical
approximation.

Another proof of correctness
appears in~\cite{bouyer-streng}, which
is based on interval arithmetic and
formulas of~\cite{lauter-viray}.
Our proof predates that proof and,
while it is more complicated, our proof
requires no numerical approximations.

\subsubsection{Relation of Theorems \ref{thm:wamelenlist}
and \ref{thm:classification2}
to our results}\label{sec:relationtoresults}

Denote by $I_{\reflexfield}(f)$ the group of fractional ideals
of $\O_{\reflexfield}$ coprime to a fixed integer~$f$.
We will first explain how its subgroup $\kernel{\O}$ relates to Theorems~\ref{thm:wamelenlist}
and~\ref{thm:classification2}.
Let $C/k$ be a curve of genus two over a number field,
and suppose that the endomorphism ring
$\End({J(C)}_{\overline{k}})$ is isomorphic to
an order $\O$ in a quartic field~$K$.

Then $K$ is a CM-field, and the theory of complex multiplication
gives us some CM-type belonging to the isomorphism
$\O\rightarrow \End({J(C)}_{\overline{k}})$.
Let $(\reflexfield, \reflextype)$ be the reflex of $(K,\Phi)$.
By the Main Theorem of Complex Multiplication
for arbitrary orders \cite[\S 17.3, Main Theorem~3]{shimura-taniyama},
the composite $k\cdot \reflexfield$ contains
the unramified class field $k_1$ of $\reflexfield$
corresponding to the ideal group~$\kernel{\O}\subset I_{\reflexfield}(f)$.

In particular, if $k=\QQ$, then $k_1=\reflexfield$,
so the inclusions $\kernel{\O}\subset \kernel{\Omax}\subset
I_{\reflexfield}(f)$ are equalities.
The following result explicitly generates a minimal order
with this property.

\begin{lemma}\label{lem:Omin}
Let $K\not\cong \QQ(\zeta_5)$
be a non-biquadratic
quartic CM-field with $\kernel{\Omax}=I_{\reflexfield}$.
Let the ideals $\aaa_1,\ldots,\aaa_n\subset\Omaxreflex$ be generators of the ray class group of $\reflexfield$ modulo~$f$,
and let $\mu_i\in K^\times$ be generators of $N_{\reflextype}(\aaa_i)$
such that $\mu_i\overline{\mu_i}\in\QQ$.

Let $\O$ be an order in $K$ such that
$f\Omax\subset \O$.
We have $\kernel{\O}=I_{\reflexfield}$
if and only if
$\O\supset \Ominf{f}:=\ZZ[\mu_i : i\in\{1,\ldots,n\}] + f\Omax$.
\end{lemma}
\begin{proof}
The $\mu_i\in\Omax$ exist as $\kernel{\Omax}=I_{\reflexfield}$,
and they are uniquely determined up to roots of unity,
hence uniquely determined up to sign as $K\not\cong\QQ(\zeta_5)$.
Since $\O$ is a ring and $\kernel{\O}=I_{\reflexfield}$, both
$\mu_i$ and $-\mu_i$ are in $\O$ for each $i$, hence
$\O$ contains $\Ominf{f}$.
Conversely, if $\O$ contains $\Ominf{f}$,
then $\kernel{\O}=I_{\reflexfield}$.
\end{proof}

\subsubsection{{The case of maximal orders (proof of Theorem~\ref{thm:wamelenlist})}}\label{sec:provingwamelen}

The first curve $y^2=x^5+1$ is well-known
to have
endomorphism ring $\ZZ[\zeta_5]$~\cite[Example 15.4.2]{shimura-taniyama},
and van Wamelen computed (as mentioned above) that it is
unique with
this property.

Next, \cite{wamelen-correctness}, or more precisely, its
data set \cite{wamelen-correctness-data}, gives, for each of
the other $12$ fields in Table~\ref{tab:provingwamelen},
an order $\Osingle$ with
a proof that
$\O:=\End({J(C)_{\overline{\QQ}}})\supset\Osingle$
holds for the curve(s) corresponding to that field.
We give these orders in Table~\ref{tab:provingwamelen}.
As $J(C)$ is principally polarised, the Rosati involution
maps $\O$ into itself, and since the Rosati involution
acts as complex conjugation on~$K=\QQ(\O)$, this implies
$\O\supset\Osingle+\overline{\Osingle}=:\Odouble$.
At the same time, as $K$ has no imaginary quadratic
subfields, the endomorphism ring $\O$
is an order in~$K$ by~\cite[Theorem~1.3.3]{lang-cm}.

Next, we take~$f$ such that
$f\Omax\subset \Odouble$ and compute
$\Ominf{f}$ as in Lemma~\ref{lem:Omin}
using Sage~\cite{sage}.
That lemma then gives $\O\supset \Odouble+\Ominf{f}=:\Otriple$, so we compute the latter ring. Note that this ring does
not depend on the CM-type $\Phi$ appearing in Lemma~\ref{lem:Omin}.
Indeed, the reflex field $\reflexfield\subset \CC$
is the unique subfield isomorphic to~$K$,
and as $\Gal(\reflexfield/\QQ)\cong C_4$,
all CM-types of $\reflexfield$ with values in~$K$
are of the form $\reflextype\circ\sigma$
with $\sigma\in \Gal(\reflexfield/\QQ)$,
so $N_{\reflextype\circ\sigma}(\aaa_i)=
N_{\reflextype}(\sigma(\aaa_i))$.

The resulting orders $\Otriple$ are equal to $\Omax$
in all but two cases. In the other two, we
use Sage~\cite{sage} to compute the 
principally polarised ideal classes of $\Otriple$ as
in~\cite[Section~4.3]{broker-lauter-streng}, and find that each of them
has CM by the maximal order.
This proves Theorem~\ref{thm:wamelenlist}.\qed

\begin{table}
\[
\begin{array}{|l|l|l|l|l|l|}
\hline
[D,A,B] & n &\chi & i_1 & i_2 & i_3
\\ \hline
[5,5,5] & 1 & & & & \\ \hline
[8, 4, 2] & 1 & x^{4} + 4 x^{2} + 2 & 1 & 1 & 1 \\  \hline 
[13, 13, 13] & 1 & x^{4} - x^{3} + 2 x^{2} + 4 x + 3 & 3 & 1 & 1 \\  \hline 
[5, 10, 20] & 2 & x^{4} + 10 x^{2} + 20 & 4 & 4 & 2 \\  \hline 
[5, 65, 845] & 2 & x^{4} - x^{3} + 16 x^{2} - 16 x + 61 & 19 & 1 & 1 \\  \hline 
[29, 29, 29] & 1 & x^{4} - x^{3} + 4 x^{2} - 20 x + 23 & 7 & 1 & 1 \\  \hline 
[5, 85, 1445] & 2 & x^{4} - x^{3} + 21 x^{2} - 21 x + 101 & 29 & 1 & 1 \\  \hline 
[37, 37, 333] & 1 & x^{4} - x^{3} + 5 x^{2} - 7 x + 49 & 21 & 3 & 1 \\  \hline 
[8, 20, 50] & 2 & x^{4} + 20 x^{2} + 50 & 25 & 25 & 1 \\  \hline 
[13, 65, 325] & 2 & x^{4} - x^{3} + 15 x^{2} + 17 x + 29 & 23 & 1 & 1 \\  \hline 
[13, 26, 52] & 2 & x^{4} + 26 x^{2} + 52 & 36 & 36 & 2 \\  \hline 
[53, 53, 53] & 1 & x^{4} - x^{3} + 7 x^{2} + 43 x + 47 & 13 & 1 & 1 \\  \hline 
[61, 61, 549] & 1 & x^{4} - x^{3} + 8 x^{2} - 42 x + 117 & 39 & 3 & 1 \\  \hline 
\end{array}\]
\caption{
In Section~\ref{sec:provingwamelen}, we define orders
$\Osingle, \Odouble, \Otriple$ of indices
$i_1,i_2,i_3$ in their normal closures as follows.
Let $\Osingle:=\ZZ[x]/(\chi)$,
then its field of fractions $K$ is isomorphic to
$\QQ[X]/(X^4+AX^2+B)$
and contains the real quadratic field of discriminant~$D$.
Van Wamelen proves that his $n$ curves corresponding
to~$K$ have endomorphism ring
containing
$\Osingle$.
We prove that this implies that the endomorphism ring
contains
$\Odouble=\Osingle+\overline{\Osingle}$
and $\Otriple=\Odouble+\Ominf{i_2}$.
}\label{tab:provingwamelen}
\end{table}

\subsubsection{The case of non-maximal orders}
\label{sec:evidence}

Next, we explain how our results and some additional computations
prove Theorem~\ref{thm:classification2}.
Details of the computations are available online~\cite{recip}
as a
Sage~\cite{sage}
file, and we give the main steps and ideas here.

We start with the completeness.
Let $C/\QQ$ satisfy $\O\cong \mathrm{End}(J(C)_{\overline{\QQ}})$
for some order $\O$ in some quartic number field~$K$.
Then as in Section~\ref{sec:relationtoresults},
we have $\kernel{\Omax}=I_{\reflexfield}$.
We took all known
cyclic quartic CM-fields with this property
from Bouyer-Streng~\cite{bouyer-streng} which gave the
$20$ fields listed in Tables~\ref{tab:provingwamelen}
and~\ref{tab:otherfields}.
\begin{table}
\[
\begin{array}{|l|}
\hline
[D,A,B]\\
\hline
  [ 5, 15, 45]\\  
\hline  
  [ 5, 30, 180]\\    
\hline [  5, 35, 245]\\    
\hline [   5, 105, 2205]\\
\hline [   8, 12, 18]\\
\hline [ 17, 119, 3332]\\
\hline [ 17, 255, 15300]\\
\hline
\end{array}
\]
\caption{The known fields with $\kernel{\Omax}=I_{\reflexfield}$
that are not in Table~\ref{tab:provingwamelen},
given by triples $[D,A,B]$ with $K=\QQ[X]/(X^4+AX^2+B)$
and $\Delta_{K_0}=D$.}\label{tab:otherfields}
\end{table}

For each of the $20-1=19$ fields~$K\not\cong \QQ(\zeta_5)$,
we did the following computations.
For each prime $p\mid 2\cdot 3\cdot N_{K_0/\QQ}(\Delta_{K/K_0})$,
we use Sage~\cite{sage}
to compute the sequence of rings
$A_k=\Ominf{p^k}$ for $k=0,1,\ldots$
until it stabilises, which we recognise as follows.
\begin{lemma}
If $A_{k+1}=A_k$, then for all $l\geq k$, we have $A_l=A_k$.
\end{lemma}
\begin{proof}
Let $A=A_{k+2}$. It suffices to prove $A=A_{k+1}$.
Note that for $n\leq k+2$, we have $A_n=A+p^{n}\Omax$.
In particular, we have $A \subset A_{k+1}=A+p^k\Omax$,
where the quotient for the inclusion is
a power of $(\ZZ/p\ZZ)$.
Therefore, multiplying on the right with~$p$
reverses the inclusion, so $A\supset pA + p^{k+1}\Omax$,
so $A\supset A+p^{k+1}\Omax$, which is
what we needed to show.
\end{proof}
For our list of $19$ fields,
it turns out that the chain always stabilises at
$p^k=1$, except in the case $p=2$ for $7$ of the fields,
where it stabilises at $2^1$ with $[\Omax:\Ominf{2}]\in\{2,4\}$.
In particular, as no odd prime power greater than one appears, we have
$\Ominf{f} = \Ominf{2^k}$
for $k\in\{0,1\}$ for all our $19$ fields~$K$.
For the $7$ fields with $k=1$,
we compute all non-maximal superorders
of $\Ominf{2}$ and their
principally polarised ideal classes using Sage.
We can check the existence of $(2,2)$-isogenies in a proven
manner on the level of these principally polarised
ideal classes, as the polarised complex tori corresponding to
$(\mathfrak{a}_1,\xi_1)$ and $(\mathfrak{a}_2,\xi_2)$
are $(\ell,\ell)$-isogenous if and only if there exists $\mu\in K^\times$ with
$\mathfrak{a}_1\subset\mu^{-1}\mathfrak{a}_2$
and $\xi_1 = \ell\mu\overline{\mu}\xi_2$.
The computations above yield the following result.
\begin{lemma}
Each principally polarised ideal class with multiplier
ring a non-maximal order~$\O$ with $\kernel{\O}=I_{\reflexfield}$
in one of our $19$ fields~$K\not\cong\QQ(\zeta_5)$
is $(2,2)$-isogenous to a unique principally polarised
ideal class of the maximal order of~$K$.\qed
\end{lemma}
This proves Theorem~\ref{thm:classification2} for the fields in Table~\ref{tab:otherfields}:
as the (unique $(2,2)$-isogenous) curves with CM by the maximal order are not stable under
$\mathrm{Gal}(\QQbar/\QQ)$, neither are the curves with CM by the non-maximal orders.

For the $13-1=12$ fields $K\not\cong\QQ(\zeta_5)$
in Table~\ref{tab:provingwamelen},
we used the AVIsogenies~\cite{avisogenies} Magma~\cite{magma}
package to compute all principally polarised abelian surfaces over~$\QQ$
that are $(2,2)$-isogenous to those of Theorem~\ref{thm:wamelenlist}.
This yielded no curves not covered by Theorem~\ref{thm:wamelenlist},
hence proves Theorem~\ref{thm:classification2} outside of the case $K=\QQ(\zeta_5)$.

This leaves the field $K=\QQ(\zeta_5)$, where Lemma~\ref{lem:Omin}
does not directly apply.
There we do have
$\O\supset \ZZ[\zeta_5^{e_i}\mu_i:i]+f\Omax$
with $0\leq e_i<5$, so the computations are still finite, but
a little more complicated. In the end, this yields 7 orders,
the indices of which all happen to be prime
powers (dividing $2^4$, $3^2$ or $5^2$ to be precise).
We compute the corresponding period matrices,
and
they all turn out to be related to $y^2=x^5+1$
by a $(3,3)$-isogeny, a $(5,5)$-isogeny or 
a chain of at most two $(2,2)$-isogenies.

In the case of the $(5,5)$-isogeny, we found a unique period matrix and evaluated the absolute Igusa
invariants in them with interval arithmetic
to an additive error less than $2^{-1}5^{-8}$.
These Igusa invariants are in~$\QQ$
and we were advised by Kristin Lauter
that the denominator bounds of Lauter and
Viray~\cite{lauter-viray} hold also for Igusa
invariants of curves with CM by non-maximal
orders of the form $\Omaxreal[\eta]$.
The relevant order is of that form
for $\eta=\zeta_5+3\zeta_5^3$
and we computed using their formulas
that the denominator of the Igusa invariants in~$\QQ$
divides~$5^8$.
Together, this proves correctness of the Igusa invariants
that we computed, hence proves correctness of the curve.

For the $(3,3)$-isogeny, we find using AVIsogenies
that over $\overline{\FF_{23}}$, there are
$40$ curves that are $(3,3)$-isogenous to $F:y^2=x^5+1$,
none of which have their moduli in $\FF_{23}$.
As these are the reductions of the $40$ curves over
$\QQbar$ that are $(3,3)$-isogenous to $F$,
we find that none of them are defined over~$\QQ$.

For the $(2,2)$-isogenies, we used the Richelot isogeny code of AVIsogenies
directly over~$\QQ$. It returned the curve $C$
from Theorem~\ref{thm:classification2}, which is $(2,2)$-isogenous
to~$F$, and no other curve over $\QQ$ that is $(2,2)$-isogenous
to~$C$ or~$F$.
This shows that~$F$
has CM by a non-maximal order $\O\supset 1+2\Omax$.
The only such order with $\kernel{\O}=I_{\reflexfield}$
for which period matrices exist is the one given in Theorem~\ref{thm:classification2}.
A Sage computation with
principally polarised ideal classes shows that every curve
with CM by an order $\O$ with $\kernel{\O}=I_{\reflexfield}$
and even index in $\Omax$ is $(2,2)$-isogenous to~$C$ or~$F$,
and we check using AVIsogenies that no such curve exists
other than $C$ and~$F$ themselves;
this proves Theorem~\ref{thm:classification2}
for the one remaining field $\QQ(\zeta_5)$.

\subsection{Computation of endomorphism rings}
\label{sec:applicationcompend}

Let $\A$ be an ordinary principally polarised abelian variety defined over a
finite field $k$ of cardinality $q$. The characteristic polynomial of its
Frobenius endomorphism $\pi$ may be computed in polynomial time in $\log(q)$
\cite{pila}; this gives the CM-field $K=\QQ(\pi)$ of which the endomorphism
ring of $\A$ is an order $\End(\A)$ containing $\ZZ[\pi,\overline\pi]$ and
stable under complex conjugation \cite{waterhouse}.

The order $\End(\A)$ is a finer invariant
than the characteristic polynomial of~$\pi$,
but until recently all known methods for 
computing it had
exponential complexity. The sub-exponential method Bisson and Sutherland
obtained for elliptic curves \cite{endomorphism,grh-only} proved to be very efficient
and has since enabled various
applications
\cite{drew-modpol,drew-hilbert}. Bisson then generalised it to abelian
varieties \cite{end-g2}, which also improved existing applications
\cite{lauter-robert}, but his generalisation relies on several unproven
heuristic assumptions. Using our results, we will show that a particular one
was false in general, but that it holds for almost all abelian surfaces of
typical families.

\subsubsection{Background}

To evaluate the endomorphism ring $\End(\A)$ of an abelian surface $\A$ such as
above, 
Bisson \cite{end-g2}
uses isogeny computation
to determine the structure of the polarised class group
$\CCC(\End(\A))$.
The endomorphism ring $\End(\A)$ is
then identified amongst orders $\O$ satisfying $\CCC(\O)=\CCC(\End(\A))$ by
computing it locally at primes that divide the index between such orders. 

Fix $f=[\Omax:\ZZ[\pi,\overline\pi]]$ and restrict to ideals coprime to $f$
so that the groups $P_\O\subset I_\O\subset
I_{\Omax}$ may be compared as $\O$ ranges through candidate endomorphism rings.
Rather than computing $\CCC(\O)=I_\O/P_\O$ by finding
many elements of $P_\O$, this method fixes an arbitrary CM-type
$\Phi$ of $K$ and only uses elements of $\norm_{\Phi^r}(\norm_\Phi(p_\O))$,
where $p_\O$ stands for the group of principal ideals of $\O$, since determining
whether such elements are trivial in $\CCC(\End(\A))$
can be done in sub-exponential time in $\log(q)$.
Then, it incurs a
polynomial cost in $v_\ell=\ell^{\val_\ell(f)}$ to compute the endomorphism
ring locally at $\ell$ using \cite{eisentrager-lauter} for each prime factor
$\ell$ of
\begin{equation}\label{eq:annoying-primes}
\lcm\left\{[\O+\Oother:\O\cap\Oother] :
\begin{array}{l}
\norm_\Phi(p_\O)\subset \kernel{\Oother} \\
\norm_\Phi(p_{\Oother})\subset \kernel{\O}
\end{array}\right\},
\end{equation}
where $\O$ and $\Oother$ range through all orders of $K$ containing
$\ZZ[\pi,\overline\pi]$ stable under complex conjugation, and the map
$\norm_\Phi$ takes an ideal of $p_\O$
(respectively~$p_{\Oother}$) to $I_{K^r}$.

At first, Bisson expected \eqref{eq:annoying-primes} to be uniformly bounded
for all quartic CM-fields~$K$; the cost of this local computation would then be
negligible. Example~\ref{ex:realindexrequired} disproves this expectation since
it gives pairs of orders with $\kernel{\O}=\kernel{\Omin}$ and arbitrarily large indices
$[\O:\Omin]$. However, we will now see that, under certain heuristics,
\eqref{eq:annoying-primes} is almost always small.

\subsubsection{Bounding the cost of local endomorphism ring computations}

Our results would apply directly if \eqref{eq:annoying-primes} had
$\norm_\Phi(p_\O)$ replaced by $\kernel{\O}$ and $\norm_{\Phi}(p_{\Oother})$ by~$\kernel{\Oother}$;
nevertheless, those two groups are closely related as the following lemma
shows.

\begin{lemma}
For any order $\O$ in a quartic non-biquadratic
CM-field $K$ and any CM-type $\Phi$ of $K$,
we have
\[
(\kernel{\O}\cap\norm_\Phi(I_\O))^2\subset\norm_\Phi(p_\O)\subset \kernel{\O}.
\]
\end{lemma}

\begin{proof}
Consider the composition
$
I_{K}   \stackrel{\norm_{\Phi  }}{\longrightarrow}
I_{K^r} \stackrel{\norm_{\Phi^r}}{\longrightarrow}
I_{\Omax}
$
where $I_K$ and $I_{K^r}$ are the groups of invertible fractional ideals of
$\Omax$ and $\Omaxreflex$ respectively, and recall from the proof of
Lemma~\ref{lemma:squares} that $\norm_{\Phi^r}\norm_{\Phi}(\aaa)
=\tau(\aaa\overline\aaa)\aaa^2$;
if $\aaa$ admits a generator coprime to $f$
in $\O$, its image through $\norm_{\Phi^r}\circ\norm_\Phi$ thus also does.
Therefore the type norm maps $p_\O$ to a subset of $\kernel{\O}$.

Let $\bbb$ lie in the intersection of $\kernel{\O}$ and $\norm_\Phi(I_\O)$. This means
that, for some $(\aaa,\alpha)\in I_\O$, we have
$\bbb=\norm_\Phi(\aaa,\alpha):=\norm_\Phi(\aaa)$
and $\norm_{\Phi^r}\norm_{\Phi}(\aaa)\in P_\O$. Equation
\eqref{eq:reflex-norm-square} then states that $(\aaa,\alpha)^2$ belongs to
$P_\O$; by composing with $\norm_\Phi$ we find that $\bbb^2$ lies in
$\norm_\Phi(P_\O)$ and hence in $\norm_\Phi(p_\O)$.
\end{proof}

Therefore $\norm_\Phi(p_\O)$ is not much different from $\kernel{\O}$; in fact,
similarly to Theorem~\ref{thm:firstmaintheorem} we have:

\begin{corollary}
Let $\O$ and $\Oother$ be two orders satisfying the conditions of
\eqref{eq:annoying-primes}. The indices $[\O:\Omin]$ and $[\Oreal:\Ominreal]$
have the same valuation at all primes $\ell>41$, and further $\ell>7$ when
$\O\not\cong\ZZ[\zeta_5]$.
\end{corollary}

\begin{proof}
By the above lemma, the conditions of \eqref{eq:annoying-primes} imply
$(\kernel{\O}\cap\norm_\Phi(I_\O))^2\subset \kernel{\Oother}$; as in
Section~\ref{sec:ideals-to-elements} we deduce
\[
\ker(\CCC(\Omin)\to\CCC(\O))^4\subset\ker(\CCC(\Omin)\to\CCC(\Oother)).
\]
Indeed, let $(\aaa,\alpha)^4$ represent a class of the first kernel. By
\eqref{eq:reflex-norm-square} we obtain
$(\aaa,\alpha)^4=\norm_{\Phi^r}\norm_\Phi(\aaa)^2$ in $\CCC(\Omin)$ and,
by assumption, $\norm_\Phi(\aaa)^2$ belongs to
$(\kernel{\O}\cap\norm_\Phi(I_\O))^2$. In particular, we have
$\norm_{\Phi^r}\norm_\Phi(\aaa)^2\in\norm_{\Phi^r}(\kernel{\Oother})$ so the class of
$(\aaa,\alpha)^4$ in $\CCC(\Omin)$ becomes trivial in $\CCC(\Oother)$.

Now, using the same proof as for Proposition~\ref{prop:kernel} (albeit
replacing two's by four's in the exponents) we deduce that, for any two orders
$\O$ and $\Oother$ satisfying the conditions of \eqref{eq:annoying-primes}, the
kernel \eqref{eq:kernel} is of exponent at most four. The proof of
Proposition~\ref{prop:valuation} then carries through for all primes $p>7$
assuming $\O\not\cong\ZZ[\zeta_5]$, and $p>41$ in general.
\end{proof}

As a consequence, we now establish that the sum $\sum_\ell v_\ell$ where $\ell$
ranges through prime factors of \eqref{eq:annoying-primes} is almost always
small, assuming that the $v_i$ satisfy typical divisibility conditions.

\begin{proposition}\label{prop:density-zero}
Let $(\A_i/\FF_{q_i})_{i\in\NN}$ be a sequence of ordinary abelian varieties
defined over fields of monotonously increasing cardinality $q_i\to\infty$. Denote by
$v_i=[\O_{\QQ(\pi_i)}:\ZZ[\pi_i,\overline\pi_i]]$ their conductor gaps, and by
$n_i=\norm_{K_0/\QQ}(\Delta_{K/K_0})$ the norm of the relative discriminant of
their CM-fields $K=\QQ(\pi_i)$. Assume that there exists a constant $C$ such
that, for all positive integers $u$ and $m$:
\begin{itemize}
\item the proportion of indices $i<m$ for which $u|v_i$ is at most $C/u$;
\item the proportion of indices $i<m$ for which $u|v_i$ and $u|n_i$ is at most $C/u^2$.
\end{itemize}
Then, for any $\tau>0$, all prime factors $\ell$ of \eqref{eq:annoying-primes}
are such that $v_{i,\ell}=\ell^{\val_\ell v_i}$ is smaller than $L(q_i)^\tau$,
except for a zero-density subset of indices $i\in\NN$,
where $L(x)=\exp\sqrt{\log x\cdot\log\log x}$.
\end{proposition}

\begin{proof}
Fix an index $i$ and let $\ell$ be a prime factor of $v_i$. If $\ell^2$ does
not divide~$v_i$, then, locally at $\ell$, out of two distinct orders containing
$\ZZ[\pi_i,\overline\pi_i]$, one must be maximal; by the proof of
Theorem~\ref{thm:relativeindex}, the quantity \eqref{eq:annoying-primes} thus
has no $\ell$-part unless $\ell\leq 41$ or $\ell|n_i$. As a consequence, for
any fixed integer $M\geq 41$, all pairs of indices $i<m$ and primes $\ell$ dividing
\eqref{eq:annoying-primes} such that $v_{i,\ell}>L(q_i)^\tau$ satisfy at
least one of the following conditions:
\begin{itemize}
\item the index $i$ satisfies $L(q_i)^\tau\leq M$;
\item $\val_\ell v_i=1$ and $\ell>M$ divides $\gcd(v_i,n_i)$;
\item $\val_\ell v_i>1$ and $v_{i,\ell}>M$ divides $v_i$.
\end{itemize}
Since $v_i<64 q_i^2$ \cite[Lemma~3.2]{end-g2} we have $v_{i,\ell}<64 q_m^2$ 
for all primes $\ell$ and indices $i<m$, thanks to the monotony of $q_i$. 
Hence the proportion of indices $i<m$ for which \eqref{eq:annoying-primes}
admits a prime factor $\ell$ such that $v_{i,\ell}>L(q_i)^\tau$ is bounded by
\[
\frac{\#\{i:L(q_i)^\tau\leq M\}}{m}+
\sum_{M<\ell<64 q_m^2}\frac{C}{\ell^2}+
\sum_{2\leq\alpha\leq6+2\log_2q_m}\left(\sum_{M^{1/\alpha}<\ell\leq M^{1/(\alpha-1)}}\frac{C}{\ell^\alpha}\right)
\]
where each term corresponds to one of the above conditions, and the variable
$\alpha$ represents $\val_\ell v_i$. Note that the bound
$\ell\leq M^{1/(\alpha-1)}$ in the last sum is here to prevent counting the same
prime $\ell$ for multiple values of $\alpha$.

By bounding each sum, we further deduce that the density of such indices
is less than
\[
\lim_{m\to\infty}\left(
\frac{\#\{i:L(q_i)^\tau<M\}}{m}+
\frac{C}{M}+
\sum_{\alpha>1}
\frac{C}{\alpha-1}\left(\frac{1}{M^{\frac{\alpha-1}{\alpha}}}-\frac{1}{M}\right)
\right)
\]
For $M$ large enough, the second and third terms can be made arbitrarily small;
then, the first term vanishes as $m$ goes to infinity.
The above density is thus zero.
\end{proof}

\subsubsection{Conclusion}

The conditions of Proposition~\ref{prop:density-zero} on the integers $v_i$ are
typical divisibility properties
that are satisfied by integers drawn uniformly at random (say,
from $\{1,\ldots,64q_i^2\}$) and independently from the~$n_i$.
Furthermore, they are experimentally observed to hold for the two main families
which are of interest to~\cite{end-g2},
namely random isomorphism classes of
abelian varieties and random isomorphism classes of abelian varieties with
complex multiplication by a prescribed field, defined over finite fields of
increasing cardinality.

The method of \cite{end-g2} rests on several unproven heuristics, including the
generalised Riemann hypothesis and that the quantity \eqref{eq:annoying-primes}
remains small. Although the latter is false in general,
Proposition~\ref{prop:density-zero} shows that it holds for almost all abelian
varieties of families such as above. Under heuristic assumptions, it therefore
establishes that the method of \cite{end-g2} terminates in average
sub-exponential time. 

\section*{Acknowledgements}

The authors would like to thank David Gruenewald,
Kristin Lauter and Bianca Viray
for helpful discussions.
The first-named author further wishes to thank the Mathematics Research Centre of the
University of Warwick for their hospitality while part of this work was done.
The second-named author acknowledges the support of EPSRC grant EP/G004870/1
and NWO through DIAMANT and Vernieuwingsimpuls.

\bibliographystyle{plain}
\bibliography{document}

\end{document}